\documentclass[10pt]{amsart}
\usepackage[utf8]{inputenc}
\usepackage{bm}
\usepackage{fontenc}
\usepackage{amsfonts}
\usepackage{amssymb}
\usepackage{amsmath}
\usepackage{amsthm}\usepackage{mathtools}
\usepackage{enumerate}
\usepackage{hyperref}\usepackage{enumitem}
\usepackage{mathrsfs}
\usepackage{tikz}\usetikzlibrary{calc}
\usepackage{marginnote}
\usepackage{xcolor,enumitem}
\usepackage{soul}\usepackage{tikz}
\usepackage[all,cmtip]{xy}
\usepackage{tikz,tikz-cd}

\newcommand{\N}{\mathbb{N}}

\newcommand{\cE}{\mathcal{E}}

\newcommand{\SOTh}{\mathrm{SOT}\text{-}}

\newcommand{\cstu}{\mathrm{C}^*_u}

\newcommand{\cstql}{\mathrm{C}^*_{ql}}

\newcommand{\roeq}{\mathrm{Q}^*_u}
\newcommand{\qlq}{\mathrm{Q}^*_{ql}}
\newtheorem*{rigprob*}{Rigidity Problem for uniform Roe Algebras}
\newtheorem*{rigprobcorona*}{Rigidity Problem for uniform Roe Coronas}

\newcommand{\propg}{\mathrm{prop}}

\newcommand{\cstar}{$\mathrm{C}^*$}

\newtheorem{theorem}{Theorem} [section]
\newtheorem*{theorem*}{Theorem}
\newtheorem{prop}[theorem]{Proposition}
\newtheorem{problem}[theorem]{Problem}
\newtheorem*{proposition*}{Proposition}
\newtheorem{lemma}[theorem]{Lemma}
\newtheorem*{lemma*}{Lemma}

\newtheorem*{corollary*}{Corollar}

\newtheorem*{fact*}{Fact}
\theoremstyle{definition}
\newtheorem{definition}[theorem]{Definition}
\newtheorem*{definition*}{Definition}
\newtheorem{claim}[theorem]{Claim}
\newtheorem*{claim*}{Claim}

\newtheorem*{conjecture*}{Conjecture}

\newtheorem{theoremi}{Theorem}

\theoremstyle{remark}

\newtheorem*{example*}{Example}

\newtheorem*{remark*}{Remark}

\newtheorem*{note*}{Note}
\newtheorem*{question*}{Question}

\newcommand{\norm}[1]{\left\lVert #1 \right\rVert}

\DeclareMathOperator{\Fin}{Fin}

\DeclareMathOperator{\Aut}{Aut}
\DeclareMathOperator{\rank}{rank}
\DeclareMathOperator{\corank}{corank}
\DeclareMathOperator{\Img}{Im}

\DeclareMathOperator{\Ad}{Ad}

\newcounter{my_enumerate_counter}
\newcommand{\pushcounter}{\setcounter{my_enumerate_counter}{\value{enumi}}}
\newcommand{\popcounter}{\setcounter{enumi}{\value{my_enumerate_counter}}}

\usepackage{enumitem}

\title{The rigidity problem for uniform Roe algebras}

\author{Alessandro Vignati}
\address[AV]{
 Institut de Math\'ematiques de Jussieu - Paris Rive Gauche (IMJ-PRG)\\
 Universit\'e Paris Cit\'e\\ Institut Universitaire de France\\
 B\^atiment Sophie Germain\\
 8 Place Aur\'elie Nemours \\ 75013 Paris, France}
\email{ale.vignati@gmail.com}
\urladdr{http://www.automorph.net/avignati}

\date{\today}

\begin{document}

\begin{abstract}
We solve the rigidity problem for uniform Roe algebras, by showing that two uniformly locally finite metric spaces with isomorphic uniform Roe algebras are bijectively coarsely equivalent.
\end{abstract}

\maketitle    
\section{Introduction}

This article focuses on uniform Roe algebras, \cstar-algebras (self-adjoint norm closed subalgebras of bounded operators on a complex Hilbert space) capable of coding in operator algebraic terms coarse geometric properties of metric spaces.

Coarse geometry is the study of metric spaces  from a large scale. The local structure of the spaces involved is irrelevant, and one typically works with discrete spaces. Here, we focus on \emph{uniformly locally finite} (u.l.f.\ from now on) spaces $(X,d_X)$, meaning that for all $r>0$
\[
\sup_{x\in X}|B_r(x)|<\infty,
\]
where $B_r(x)$ is the ball of radius $r$ centered at $x$, and $|B_r(x)|$ is its cardinality. U.l.f.\ spaces are often referred to as spaces \emph{of bounded geometry}.
Examples of u.l.f.\ metric spaces important for applications are (Cayley graphs of) finitely generated groups with the word metric and discretisations of Riemannian manifolds. 
Historically, coarse geometric ideals appeared in Mostow's rigidity theorem as well as in early work on growth of groups. The impetus came from Gromov's revolution of geometric group theory (e.g. \cite{Gromov_ICM84,Gromov93}). Coarse geometry found many important applications in several areas of mathematics, for example (higher) index theory (e.g. \cite{Roe1993,Yu.NovikovFAD,Yu_ICM06}), topological rigidity of manifolds (e.g. \cite{GuentnerTesseraYu2012,Guentner:2013aa}), and Banach space theory (e.g. \cite{Ostrovskii:2009nx,DraniGongLaffYu}). We recommend the excellent book \cite{Nowak:2012aa} or the survey \cite{Yu.Large} for an introduction to coarse geometry. Relevant functions in this area are coarse maps and equivalences.

\begin{definition}
Let $(X,d_X)$ and $(Y,d_Y)$ be metric spaces. A function $f\colon X\to Y$ is  \emph{coarse} if for every $r>0$ there is $s>0$ such that for all $x,x'\in X$
\[
\text{ if }d_X(x,x')\leq r \text{ then }d_Y(f(x),f(x'))\leq s.
\]
If $f,g\colon X\to Y$ are two maps, we say that $f$ and $g$ are \emph{close}, and write $f\sim_{cl} g$, if 
\[
\sup_{x\in X}d_Y(f(x),g(x))<\infty.
\]
If $f\colon X\to Y$ and $g\colon Y\to X$ are coarse functions such that
\[
f\circ g\sim_{cl}Id_Y \text{ and } g\circ f\sim_{cl}Id_X,
\]
then each of $f$ and $g$ is called a \emph{coarse equivalence}, $f$ and $g$ are called \emph{mutual inverses}, and the spaces $X$ and $Y$ are said to be \emph{coarsely equivalent}. If in addition $f$ is a bijection, it is called a \emph{bijective coarse equivalence}, and the spaces are said to be \emph{bijectively coarsely equivalent}.
\end{definition}

John Roe defined in the late 1980s (\cite{Roe:1988qy}) a family of \cstar-algebras  capable of coding in operator algebraic terms the coarse behaviour of metric spaces.  These algebras were defined to study index theory of elliptic operators on noncompact manifolds, and nowadays take the name of \emph{Roe-like} algebras. They are fundamental in the formulation of the coarse Baum--Connes conjectures (asking for a certain assembly map to be an isomorphism) and consequently the Novikov conjecture in high-dimensional topology (e.g. \cite{higson2002counterexamples, Yu1995,Yu2000,HigsonRoe1995,Yu.NovikovFAD} and \cite{Guentner:2022hc}) and found profound applications in index theory (e.g. \cite{Spakula:2009tg,Engel:2018vm}), \cstar-algebra theory (e.g. \cite{Rordam:2010kx,Li:2017ac}), single operator theory (e.g. \cite{Rabinovich:2004xe,SpakulaWillett2017}), topological dynamics (e.g. \cite{Kellerhals:2013aa,Brodzki:2015kb}), and mathematical physics (e.g. \cite{Cedzich:2018wx,Kubota2017}). 

Chief among Roe-like algebras is the \emph{uniform Roe algebra}.  Let $(X,d)$ be a u.l.f.\ metric space. The \emph{propagation} of an $X$-by-$X$-matrix of complex numbers $a=[a_{x,x'}]_{x,x'\in X}$ is the quantity
\[
\propg(a)=\sup \{d(x,x')\mid a_{x,x'}\neq 0\}.
\]
If $a$ has finite propagation and uniformly bounded entries, by uniform local finiteness $a$ defines a bounded linear operator on the complex Hilbert space $\ell_2(X)$. Finite propagation operators form a $^*$-algebra, which we now close.
\begin{definition}
Let $(X,d)$ be a u.l.f.\ metric space. The \emph{uniform Roe algebra} of $X$, denoted by $\cstu(X)$, is the norm closure of the algebra of finite propagation operators.
\end{definition}

Other important Roe-like algebras are the Roe algebra $\mathrm{C}^*(X)$ (constructed by considering locally compact finite propagation operators on $\ell_2(X,H)$, where $H$ is a complex infinite-dimensional separable Hilbert space), the quasi-local algebra $\cstql(X)$ (Definition~\ref{def:quasi-local}), and quotient Roe-like algebras such as the uniform Roe corona $\roeq(X)$ and the Higson corona $\mathrm{C}_\nu(X)$ (Definition~\ref{def:Higsoncorona}). Each of these is capable of coding different coarse information about the spaces of interest.

An important area of research aims to build a vocabulary between coarse geometry and Roe-like algebras, and thus to understand how much geometry is remembered by these \cstar-algebras. The centrepiece of this program is represented by \emph{rigidity problems}, asking for a correspondence between isomorphisms in the coarse category and those in the category of \cstar-algebras. The following diagram summarises well-known implications:
\[
\begin{tikzcd}[arrows=Rightarrow]
X\sim_{bij.c.eq.}Y\arrow[r]\arrow[d]&\cstu(X) \cong \cstu(Y)\arrow[d] \\
 X \sim_{c.eq.} Y\arrow [r]&\mathrm{C}^*(X)\cong \mathrm{C}^*(Y).
\end{tikzcd}
\]
Here the equivalence relations $\sim_{c.eq.}$ and $\sim_{bij.c.eq.}$ denote `being coarsely equivalent' and `being bijectively coarsely equivalent', respectively. The rigidity problems ask for the horizontal implications in the above diagram to be reversed.

\begin{problem}[Rigidity of (uniform) Roe algebras]\label{prob:main}
Prove that
\begin{itemize}
\item if two uniform Roe algebras associated to u.l.f.\ metric spaces are $^*$-isomorphic, then the underlying spaces are bijectively coarsely equivalent;
\item if two Roe algebras associated to u.l.f.\ metric spaces are $^*$-isomorphic, then the underlying spaces are coarsely equivalent.
\end{itemize}
\end{problem}

A similar diagram can be drawn for Roe-like algebras arising as quotients, like uniform Roe  and Higson coronas; in this case, reverse implications (and thus solutions to the associated rigidity problems) are subject to the set theoretic ambient. We direct the reader to  \cite{BragaFarahV.Roecoronas}, \cite{BrianFarah} and \cite{Vignati.Higson} (see also the introduction of \cite{Vignati.HDR} or of \cite[\S10]{CoronaRigidity}) for a thorough discussion on quotient Roe-like algebras and the associated rigidity problems. 

Much research has been dedicated to our rigidity problems over the last 15 years. At first, the rigidity problems were analysed for property A spaces. Property A is a well-behavedness condition on the metric spaces of interest, introduced and studied by Yu for Baum--Connes purposes. Property A is an `amenability' like condition (for example, in case of finitely generated groups, property A is equivalent to exactness, and therefore all amenable groups have property A), and it is equivalent to many regularity properties stated in either algebraic or geometric terms. Notably, property A is equivalent to the uniform Roe algebra being a nuclear (equivalently amenable) \cstar-algebra (e.g. \cite[Theorem 5.3]{SkandalisTuYu2002} or \cite[Theorem 5.5.7]{BrownOzawa}). In the seminal \cite{SpakulaWillett2013AdvMath} \v{S}pakula and Willett showed that in presence of property A if two Roe algebras are isomorphic then the underlying spaces are coarsely equivalent, solving the rigidity problem for Roe algebras in this setting. Again in the realm of property A spaces, the rigidity problem for uniform Roe algebras was solved in
\cite{WhiteWillett2017}, with a very sharp use of the technical Operator Norm Localisation property (ONL, an equivalent of property A as showed in \cite{Sako2014}) of Chen, Tessera, Wang, and Yu (\cite{ChenTesseraWangYu2008}). Incremental progress in weakening the hypotheses on the spaces involved, and at the same time in developing key ideas and techniques was made in \cite{BragaFarah2018}, which contains remarkable results on uniform approximability, and other relevant results  were proved in \cite{LiSpakulaZhang2020,BragaFarahVignati2019} and \cite{BragaFarahVignati2020AnnInstFour}. 

A breakthrough was made in \cite{SquareInventiones}. There, the authors showed that, unconditionally on the geometry of the spaces of interest, isomorphism of uniform Roe algebras implies coarse equivalence between the associated u.l.f.\ metric spaces, and that the latter is equivalent to the uniform Roe algebras being Morita equivalent. Inspired by this work, in \cite{MartinezVigolo} Mart\'inez and Vigolo fully solved the rigidity problem for Roe algebras, giving at the same time a neater and simpler proof of the results of \cite{SquareInventiones}. Lastly, in the recent \cite{SquareBij}, we solved the rigidity problem for uniform Roe algebras for coarse disjoint unions of expander graphs, prototypical spaces to which previous considerations did not apply. For this result, we used methods of uniformly finite homology as developed in \cite{Block:1992qp} and \cite{Whyte99}, already considered in \cite{WhiteWillett2017} to obtain intermediate results for nonamenable metric spaces.

In this article, we give the ultimate solution to rigidity problems.
\begin{theoremi}\label{thmi:main}
Let $(X,d_X)$ and $(Y,d_Y)$ be uniformly locally finite metric spaces. If the uniform Roe algebras $\cstu(X)$ and $\cstu(Y)$ are isomorphic, then $X$ and $Y$ are bijectively coarsely equivalent.
\end{theoremi}
Our proof gives the same result for isomorphisms between quasi-local algebras.
We postpone the discussion of the proof of Theorem~\ref{thmi:main} to \S\ref{ss:strat}, and list a few corollaries of our main result which generalise results from \cite{WhiteWillett2017} and \cite{BragaVignatiGelfand} outside of the property A setting. 

The first corollary is about the uniqueness of certain Cartan subalgebras in uniform Roe algebras. Cartan subalgebras are present in the study of operator algebras since the seminal work of Murray and von Neumann. The formal notion of Cartan subalgebra was defined in the von Neumann setting in \cite{Vershik:1971aa} to abstract the concrete properties of the inclusion $L_\infty(X,\mu)\subseteq L_\infty(X,\mu)\rtimes G$ when $G$ is a group acting on a measure space $(X,\mu)$. Cartan algebras were brought to the \cstar-setting by the work of Renault (\cite{Renault2008Cartan}), building on Kumjian's \cstar-diagonals (\cite{kumjian1986c}), to model \cstar-algebraically the inclusion of the unit space of a groupoid $G$ into the groupoid itself. 

Cartan subalgebras proved to be fundamental tools for the development of the theory of both von Neumann and \cstar-algebras. A key goal is often to prove existence and uniqueness (up to isomorphism, or even up to unitary equivalence) theorems. To name a few important results linked to Cartan subalgebras, in the von Neumann setting the quest for existence and uniqueness theorems largely motivated  Popa's rigidity/deformation theory, see e.g. \cite{Ozawa:2010sz,Popa:2007fk,Popa:2014cu,Popa:2014jq}. In the \cstar-setting existence of Cartan subalgebras (in addition to nuclearity) implies the algebras involved are isomorphic to groupoid \cstar-algebras and consequently satisfy the Universal Coefficient Theorem of Rosenberg and Schochet (\cite{Rosenberg:1987bh}); this ties Cartan subalgebras to the classification programme (\cite{Li:2018yv,LiRenault,Carlsen:2017aa}). 

The study of Cartan subalgebras in uniform Roe algebras was formally initiated in \cite{WhiteWillett2017}. For a u.l.f.\ metric space $(X,d)$, the uniform Roe algebra $\cstu(X)$ has a canonical Cartan subalgebra given by $\ell_\infty(X)$, the algebra of of propagation zero operators (those diagonalised by the basis of $\ell_2(X)$ consisting of Dirac delta functions on elements of $X$, ($\delta_x)_{x\in X}$). Abstracting the algebraic properties of the inclusion $\ell_\infty(X)\subseteq\cstu(X)$ led to the definition of Roe Cartan subalgebra (see Definition~\ref{def:RoeCartan}). One of the main results of \cite{WhiteWillett2017} (Theorem B in there) shows that every Roe Cartan pair (that is, an inclusion of \cstar-algebras $A\subseteq B$ where $A$ is a Roe Cartan subalgebra in $B$) is isomorphic to the inclusion $\ell_\infty(X)\subseteq\cstu(X)$ for some u.l.f.\ metric space $X$. As for uniqueness, even though uniform Roe algebras might contain Cartan subalgebras not isomorphic to $\ell_\infty$ (see \cite[\S3]{WhiteWillett2017}, White and Willett showed in \cite[Theorem E]{WhiteWillett2017} that for property A spaces every Roe Cartan subalgebra of $\cstu(X)$ is unitarily equivalent to $\ell_\infty(X)$. This is the strongest uniqueness result available, which we generalise unconditionally on the geometry of the spaces of interest.

\begin{theoremi}\label{thmi:main2}
Let $X$ be a u.l.f.\ metric space and suppose that $A\subseteq \cstu(X)$ is a Roe Cartan subalgebra. Then there is a unitary $v\in\cstu(X)$ such that $vAv^*=\ell_\infty(X)$.
\end{theoremi}

Our second corollary generalises one of the main results of \cite{BragaVignatiGelfand}. There, Braga and the author were able to use the content of \cite{WhiteWillett2017} to prove a Gelfand-duality type theorem which links $\mathrm{Out}(\cstu(X))$, the group of outer automorphisms of  the uniform Roe algebra (that is, automorphisms of $\cstu(X)$ modulo inner ones) with $\mathrm{BijCoa}(X)$, the group of bijective coarse equivalences of $X$ modulo closeness. Given a bijective coarse equivalence $f\colon X\to X$ one associates canonically an automorphism of $\cstu(X)$ by permuting the standard basis of $\ell_2(X)$. This association induces an injective group homomorphism
\[
\mathrm{BijCoa}(X)\to\mathrm{Out}(\cstu(X)),
\]
which in the property A setting is showed to be an isomorphism (\cite[Theorem A]{BragaVignatiGelfand}). Once again, we remove all geometric constraints.
\begin{theoremi}\label{thmi:main3}
Let $X$ be a u.l.f.\ metric space. Then the canonical homomorphism
\[
\mathrm{BijCoa}(X)\to\mathrm{Out}(\cstu(X))
\]
is an isomorphism.
\end{theoremi}

The article is structured as follows: in \S\ref{ss:strat} we sketch the strategy of the proof of our main result. In \S\ref{S.HigsonPrel} we formally introduce the Higson corona, one of our main tools, and, after a few preparatory lemmas, we show that an isomorphism between two uniform Roe algebras $\cstu(X)$ and $\cstu(Y)$ must send certain `flattened' indicator functions close to $\ell_\infty(Y)$, the canonical Cartan masa of $\cstu(Y)$. This is one of the main ideas that will be used in the proof of Theorem~\ref{thmi:main}, which is contained in \S\ref{S.mainproof}. In \S\ref{S.CorConcl} we focus on Cartan subalgebras in uniform Roe algebras and automorphisms thereof, proving Theorems~\ref{thmi:main2} and \ref{thmi:main3}. The article ends with concluding remarks and open questions.

\subsection{The strategy}\label{ss:strat}
We describe the main ideas behind the proof of Theorem~\ref{thmi:main}. Fix two u.l.f.\ metric spaces $X$ and $Y$, and let $\Phi\colon \cstu(X)\to\cstu(Y)$ be a  $^*$-isomorphism. With the aid of $\Phi$, we construct a function
\[
\alpha\colon X\to \Fin(Y),
\]
where $\Fin(Y)$ is the set of all finite subsets of $Y$,  in such a way that any function $f\colon X\to Y$ with the property that $f(x)\in \alpha(x)$ for all $x\in X$ is a coarse equivalence. This is the approach of \cite{SpakulaWillett2013AdvMath} (and \cite{SquareInventiones}, and \cite{MartinezVigolo}). White and Willett in \cite{WhiteWillett2017} noticed that, provided that $\alpha$ is constructed in such a way that for every $A\subseteq X$
\begin{equation}\label{eq:Hall}\tag{$\ast$}
|A|\leq \Big|\bigcup_{x\in A}\alpha(x)\Big|,
\end{equation}
one can then apply Hall's marriage theorem and obtain an injective coarse equivalence $f\colon X\to Y$. Symmetrically, one constructs (using $\Phi^{-1}$) an injective coarse equivalence $g\colon Y\to X$. The construction ensures that $f$ and $g$ are mutual coarse inverses, and a simple application of the proof of Cantor--Schr\"oder--Bernstein theorem gives a bijective coarse equivalence. 

The main point is thus to construct a function $\alpha\colon X\to \Fin(Y)$ satisfying Equation~\eqref{eq:Hall}.
In \cite{WhiteWillett2017}, the validity of Equation~\eqref{eq:Hall}  was ensured by proving a specific norm estimate (see Equation~\eqref{thegoal} below). Such norm estimate was shown to hold with  the aid of the technical ONL, an equivalent of property A. Even though a simpler proof of the validity of such norm estimate, and consequently of Equation~\eqref{eq:Hall}, is given in \cite[Proposition 3.16]{Vignati.HDR}  (still in the property A setting), the proofs of \cite[Lemmas 6.6 and 6.8]{WhiteWillett2017} and \cite[Proposition 3.16]{Vignati.HDR} are not replicable without assuming some form of regularity for the spaces of interest. (An unsuccessful attempt led to the results of \cite{BragaFarahVignati.ONL}.) One of the main technical arguments of this article aims to prove the aforementioned norm estimate; this is achieved with a completely new approach, focused on the algebra of slowly oscillating functions and the Higson corona.

The main new idea is to leverage that an isomorphism between $\cstu(X)$ and $\cstu(Y)$ induces an isomorphism between the Higson coronas $\mathrm{C}_{\nu}(X)$ and $\mathrm{C}_{\nu}(Y)$, as the Higson corona is the center of the uniform Roe corona, as proved in \cite[Proposition 3.6]{SquareEmb}. This observation allows to show that the isomorphism $\Phi$ sends slowly oscillating functions on $X$ to compact perturbations of slowly oscillating functions on $Y$. A diagonalisation argument (inspired by the fine analysis of Higson coronas conducted in \cite[\S3]{Vignati.Higson}) then gives that specific `flattened' indicator functions in $\ell_\infty(X)$ must be sent close to $\ell_\infty(Y)$ (Proposition~\ref{prop:Higson1}). This fact is then sharply used in Lemma~\ref{lemma:itholds} to obtain the required norm estimate and consequently prove that our guessing function $\alpha$ can be constructed to satisfy Equation~\eqref{eq:Hall}.

We stress that this approach is completely new, and it has the potential to be applied to other important questions in the realm of Roe-like algebras. For example, it is conceivable this approach can be used to study the relationship between embedding between uniform Roe algebras and coarse embeddings between the associated u.l.f.\ metric spaces (see \cite{BragaFarahVignati2019}). More importantly, we believe that (refinements of) these techniques may be relevant for the \emph{dimension problem}, asking whether coarse geometric and \cstar-algebraic appropriate notions of dimension translate properly (see e.g. \cite[\S8]{Winter:2010eb} and \cite[\S7]{li2023diagonal}).

\subsection*{Acknowledgements}
The author is supported by a grant from the Institut Universitaire de France and by the ANR grant ROAR (ANR-25-CE40-5029). The author is indebted to Bruno Braga, Ilijas Farah and Rufus Willett for comments on an early version of this manuscript.

\section{Slowly oscillating functions and their images}\label{S.HigsonPrel}

We record here a few preparatory results on slowly oscillating functions and their images under isomorphisms of uniform Roe algebras. 

The first lemma is about conditional expectations onto atomic maximal abelian subalgebras of $\mathcal B(H)$ where $H$ is a complex infinite-dimensional separable Hilbert space (all Hilbert spaces are complex from now on. Fix such $H$, and let $\bar e=(e_n)_{n\in\N}$ be an orthonormal basis for $H$. Let $p_n\in\mathcal B(H)$ be the rank one orthogonal projection onto $\mathbb Ce_n$. The algebra of operators diagonalised by $\bar e$, denoted by  $D(\bar e)$, is a maximal abelian subalgebra of $\mathcal B(H)$ which is generated, as a von Neumann algebra (i.e., a weakly closed \cstar-subalgebra of $\mathcal B(H)$) by the projections $\{p_n\}_{n\in\N}$. $D(\bar e)$ is isomorphic to $\ell_\infty(\N)$. We let $\mathbb E_{\bar e}\colon \mathcal B(H)\to D(\bar e)$ be the canonical conditional expectation defined by 
\[
\mathbb E_{\bar e}(a)=\sum_n p_nap_n.
\]
$\mathbb E_{\bar e}$ is a completely positive and contractive map, and $\mathbb E_{\bar e}(a)=a$ if and only if $a\in D_{\bar e}$. (For more on conditional expectations, see \cite[II.6.10]{Blackadar.OA}).

Recall that the strong operator topology on $\mathcal B(H)$ is given by pointwise norm-convergence, that is, a net of operators $T_i$ converges strongly to $T$, written $T=\SOTh (T_i)$, if $T_i\xi$ converges in norm to $T\xi$ for every $\xi\in H$. A sequence strongly converging to $0$ is called $\SOTh$null.
 
\begin{lemma}\label{lem:SOT}
Let $\bar e=(e_n)$ be an orthonormal basis for $H$, and let $(p_n)$ be the associated sequence of rank one projections. Let $(b_n)_{n\in\N}$ be an $\SOTh$null sequence of compact operators such that for every $M\subseteq\N$ we have that $\sum_{n\in M}b_n$ exists and belongs to $D(\bar e)+\mathcal K(H)$.
Then
\[
\lim_n\norm{\mathbb E_{\bar e}(b_n)-b_n}=0.
\]
\end{lemma}
\begin{proof}
By contradiction, suppose that $(b_n)_{n\in\N}$ is a sequence as in the hypotheses for which the thesis fails for the orthonormal basis $\bar e$. For $M\subseteq\N$ we let $p_M=\sum_{n\in M}p_n$. By passing to a subsequence we can assume that 
\begin{equation}\label{eqSOT}
\inf_n\norm{\mathbb E_{\bar e}(b_n)-b_n}>0.
\end{equation}
We now construct two sequences $(G_k)_{k\in\N}$ and $(n_k)_{k\in\N}$ such that
\begin{itemize}
    \item $(n_k)_{k\in\N}$ is an increasing sequence of naturals and $(G_k)$ is a sequence of pairwise disjoint finite subsets of $\N$ such that 
    \[
    \norm{p_{G_k}b_{n_k}p_{G_k}-b_{n_k}}<2^{-k},
    \]
    and
    \item for every $k\neq k'$
    \[
        \norm{p_{G_k}b_{n_{k'}}p_{G_k}}<2^{-k-k'}.
    \]
\end{itemize}
The construction proceeds by induction. Let $n_0=0$. Since $b_0$ is compact, we can find $G_0\Subset\N$ such that $\norm{p_{G_0}b_0p_{G_0}-b_0}<1$. Suppose that $G_i$ and $n_i$ have been constructed for all $i<k$. Since $(b_n)$ is $\SOTh$null and $F=\bigcup_{i<k}G_i$ is finite, we can find a sufficiently large $n_k\in\N$ such that $\norm{p_Fb_m}<2^{-2k}$ for all $m\geq n_k$. Note that 
\[
\norm{p_{\N\setminus F}b_{n_k}p_{\N\setminus F}-b_{n_k}}<2^{-2k},
\]
and that for every $i<k$ we have that $\norm{p_{G_i}b_{n_k}p_{G_i}}\leq \norm{p_Fb_{n_k}p_F}<2^{-2k}$. Letting  $G_k\in\Fin(\N\setminus F)$ be such that
\[
\norm{p_{G_k}b_{n_k}p_{G_k}}<2^{-k}
\]
concludes the construction.

Let now 
\[
b=\sum_k p_{G_k}b_{n_k}p_{G_k}.
\]
Note that $b-\sum_kb_{n_k}$ is compact. Since $\sum_kb_{n_k}\in D(\bar e)+\mathcal K(H)$, then $b\in D(\bar e)+\mathcal K(H)$.
Let $a\in\mathcal K(H)$ such that $b-a\in D(\bar e)$. Since $b-a\in D(\bar e)$, then for all $k$
\[
\norm{\mathbb E_{\bar e}(\chi_{G_k}(b-a)\chi_{G_k})-\chi_{G_k}(b-a)\chi_{G_k}}=0.
\]
Since $a$ is compact $\lim_k\norm{\chi_{G_k}a\chi_{G_k}}=0$, and so
\begin{eqnarray*}
\lim_k\norm{\mathbb E_{\bar e}(b_{n_k})-b_{n_k}}&=&\lim_k\norm{\mathbb E_{\bar e}(\chi_{G_k}b_{n_k}\chi_{G_k})-\chi_{G_k}b_{n_k}\chi_{G_k}}\\&=&\lim_k\norm{\mathbb E_{\bar e}(\chi_{G_k}(b-a)\chi_{G_k})-\chi_{G_k}(b-a)\chi_{G_k}}=0.
\end{eqnarray*}
This contradicts \eqref{eqSOT} and concludes the proof.
\end{proof}

If $(X,d_X)$ is  a u.l.f.\ metric space, we consider the canonical basis given by indicator functions $(\delta_x)_{x\in X}\subseteq\mathcal B(\ell_2(X))$. Operators diagonalised by $\bar \delta=(\delta_x)_{x\in X}$ are identified with $\ell_\infty(X)$. These are exactly operators of propagation zero, and thus $\ell_\infty(X)$ sits inside $\cstu(X)$ as a maximal abelian self-adjoint subalgebra (\emph{masa} from now on).

Following standard notation, if $x\in X$, we write $\chi_x$ for the rank one projection onto $\mathbb C\delta_x$.  If $A\subseteq X$, we write $\chi_A$ for $\sum_{x\in A}\chi_x$. The masa $\ell_\infty(X)$ is generated as a \cstar-algebra by the projections $\{\chi_A\}_{A\subseteq X}$.

The unital inclusion $\ell_\infty(X)\subseteq\cstu(X)$ is a \emph{Cartan inclusion}, meaning that
\begin{enumerate}
    \item There is a conditional expectation $\mathbb E_X\colon \cstu(X)\to\ell_\infty(X)$ (obtained by restricting the canonical expectation $\mathbb E_{(\delta_x)}$) and
    \item $\cstu(X)$ is generated a \cstar-algebra by the normalizer of $\ell_\infty(X)$ in $\cstu(X)$, defined as
    \[
    \mathcal N_{\cstu(X)}(\ell_\infty)=\{a\in\cstu(X)\mid a\ell_\infty(X)a^* \cup a^*\ell_\infty(X)a\subseteq\ell_\infty(X)\}.
    \]
\end{enumerate}
We call $\ell_\infty(X)$ the \emph{canonical Cartan masa} of $\cstu(X)$. 

Before continuing, let us give the definition of two relevant Roe-like algebras.
\begin{definition}\label{def:quasi-local}
Let $(X,d_X)$ be a u.l.f.\ metric space. An operator $a\in\mathcal B(\ell_2(X))$ is \emph{quasi-local} if for every $\varepsilon>0$ there is $r>0$ such that $\norm{\chi_A a\chi_B}<\varepsilon$ for every $A,B\subseteq X$ such that $d_X(A,B)\geq r$. We denote by $\cstql(X)$ the algebra of quasi-local operators. 
\end{definition}
As finite propagation operators are quasi-local, $\cstu(X)\subseteq\cstql(X)$.

\begin{definition}\label{def:Higsoncorona}
Let $(X,d_X)$ be a metric space. A bounded function $f\colon X\to\mathbb C$ is said to be \emph{slowly oscillating}\footnote{Slowly oscillating functions are also called `of bounded variation' or Higson functions.} if for every $\varepsilon>0$ and $r>0$ there is a compact $F\subseteq X$ such that for every $x,x'\notin F$ we have that
\[
\text{ if }d_X(x,x')\leq r \text{ then } |f(x)-f(x')|<\varepsilon.
\]
Slowly oscillating functions form a \cstar-subalgebra of $\ell_\infty(X)$ containing $C_0(X)$, denoted by $C_h(X)$. The \emph{Higson corona} of $X$, denoted by $C_\nu(X)$, is the quotient \cstar-algebra $C_h(X)/C_0(X)$.
\end{definition}

If $X$ is a u.l.f.\ metric space, compact sets correspond to finite ones. In this case, since $\ell_\infty(X)\cap\mathcal K(\ell_2(X))=C_0(X)$, we can see the Higson corona $C_\nu(X)$ as a subalgebra of $\ell_\infty(X)/C_0(X)$, and consequently of both the \emph{uniform Roe corona} and the \emph{quasi-local corona}, that is, the quotient algebras defined as 
\[
\roeq(X)=\cstu(X)/\mathcal K(\ell_2(X)) \text{ and }\qlq(X)=\cstql(X)/\mathcal K(\ell_2(X)).
\]
The following was proved as Proposition 3.6 in \cite{SquareEmb}.
\begin{prop}\label{prop:Higsoncenter}
Let $(X,d)$ be a u.l.f.\ metric space. Then the Higson corona $C_\nu(X)$ is the center of  both $\roeq(X)$ and $\qlq(X)$.
\end{prop}

We now combine Lemma~\ref{lem:SOT} and Proposition~\ref{prop:Higsoncenter}, together with the fact that $^*$-isomorphisms between uniform Roe (and quasi-local) algebras send compact operators to compact operators and strongly continuous (\cite[Lemma 3.1]{SpakulaWillett2013AdvMath}).

We write $\pi_X\colon \mathcal B(\ell_2(X))\to\mathcal B(\ell_2(X))/\mathcal K(\ell_2(X))$ for the canonical Calkin algebra quotient map, so that
\[
\roeq(X)=\pi_X[\cstu(X)] \text{ and }\qlq(X)=\pi_X[\cstql(X)].
\]

\begin{prop}\label{prop:Higson1}
Let $(X,d_X)$ and $(Y,d_Y)$ be u.l.f.\ metric spaces, and let $\Phi\colon\cstu(X)\to\cstu(Y)$ be a $^*$-isomorphism. Suppose that $(a_n)_{n\in\N}\subseteq C_0(X)$ is a sequence of mutually orthogonal contractions such that for every $M\subseteq\N$, $\sum_{n\in M}a_n\in C_h(X)$. Then for every $\varepsilon>0$ there exists $n_0\in\N$ such that for every $n\geq n_0$
\[
\norm{\mathbb E_Y(\Phi(a_n))-\Phi(a_n)}<\varepsilon.
\]
The same statement applies to an isomorphism between quasi-local algebras.
\end{prop}
\begin{proof}
Since $(a_n)_{n\in\N}$ is a sequence of mutually orthogonal contractions in $C_0(X)$, it is $\SOTh$null. Since $\Phi$ sends compacts to compacts and it is strongly continuous, each $\Phi(a_n)$ is compact and $(\Phi(a_n))_{n\in\N}$ is $\SOTh$null. 
\begin{claim}
    For every $M\subseteq\N$ we have that $\sum_{n\in M}\Phi(a_n)\in C_h(Y)+\mathcal K(\ell_2(Y))$.
\end{claim}
\begin{proof}
   Fix $M\subseteq\N$. Since $\sum_{n\in M}a_n\in C_h(X)$, then $\pi_X(\sum_{n\in M}a_n)\in C_\nu(X)$. Since $\Phi$ maps compacts to compacts (and so does $\Phi^{-1}$), then $\Phi$ induces a $^*$-isomorphism
\[
\tilde\Phi\colon \roeq(X)\to\roeq(Y).
\]
By Proposition~\ref{prop:Higsoncenter}, $\tilde\Phi$ maps $\mathcal Z(\roeq(X))$, the center of $\roeq(X)$, to $\mathcal Z(\roeq(Y))$, which implies that 
\[
\tilde\Phi(\pi_X(\sum_{n\in M}a_n))\in C_\nu(Y),
\]
and so 
\[
\sum_{n\in M}\Phi(a_n)=\Phi(\sum_{n\in M}a_n)\in C_h(Y)+\mathcal K(\ell_2(Y)),
\]
where the first equality follows again from strong continuity of $\Phi$.  
\end{proof}
Since $C_h(Y)\subseteq\ell_\infty(Y)$ and applying Lemma~\ref{lem:SOT} to $H=\ell_2(Y)$, $\bar e$ the canonical basis $(\delta_y)_{y\in Y}$, and $b_n=\Phi(a_n)$ gives the thesis.

The same argument works when replacing uniform Roe coronas with quasi-local ones.
\end{proof}

 Let $(X,d_X)$ be a metric space, and let $A\subseteq X$. If $r\geq 0$, we write $B_r(A)$ for the ball of radius $r$ around $A$, that is
\[
B_r(A)=\bigcup_{x\in A}B_r(x).
\]
For $r>0$ and $A\subseteq X$, we write $g_{A,r}$ for the function in $\ell_\infty(X)$ defined as
\[
g_{A,r}(x)=\max\left \{0, 1-\frac{d_X(x,A)}{r}\right \}.
\]
$g_{A,r}$ is $1$ on $A$, vanishes outside of $B_r(A)$, and slowly decreases as $x$ gets further and further from $A$.

\begin{prop}\label{prop:Higson2}
Let $(X,d_X)$ and $(Y,d_Y)$ be u.l.f.\ metric spaces, and suppose that $\Phi\colon\cstu(X)\to\cstu(Y)$ a $^*$-isomorphism. Let $\varepsilon>0$. There is $m>0$ and $F\in\Fin(X)$ such that for every $A\in\Fin(X\setminus F)$ we have that
\[
\norm{\mathbb E_Y(\Phi(g_{A,m}))-\Phi(g_{A,m})}<\varepsilon.
\]
\end{prop}
\begin{proof}
The idea is to assume the thesis fails, and construct diagonally a counterexample to Proposition~\ref{prop:Higson1}. Fix $\varepsilon>0$ witnessing the failure of the thesis, and, by induction, construct a sequence of mutually disjoint finite subsets of  $X$, $(A_n)$, such that 
\[
\norm{\mathbb E_Y(\Phi(g_{A_n,n}))-\Phi(g_{A_n,n})}\geq \varepsilon.
\]
By passing to a subsequence, using local finiteness of $X$, we can assume that $d_X(A_n,A_m)\geq 2(n+m)$. For $M\subseteq\mathbb N$, define
\[
g_M=\sum_{n\in M}g_{A_n,n}.
\]
\begin{claim}
For every $M\subseteq\N$, $g_M\in C_h(X)$.
\end{claim}
\begin{proof}
Fix positive reals $\delta$ and $r$ with $\delta<1$. Since $g_M\in\ell_\infty(X)$, we can treat it is a bounded function $g_M\colon X\to\mathbb C$. Let $n\in \N$ with $n>\frac{r}{\delta}$, and $F=B_n(\bigcup_{m\leq n}A_m)$. Since $\delta<1$, $n>r$. 

We want to show that if $x,x'\in X\setminus F$ are such that $d_X(x,x')\leq r$ then $|g_M(x)-g_M(x')|<\delta$. If both $x$ and $x'$ belong to $X\setminus \bigcup_k B_k(A_k)$, then $g_M(x)=g_M(x')=0$, and there is nothing to prove. Suppose then that $x\in B_k(A_k)$ for some $k> n$, so that $x'\in B_{r+k}(A_k)$. Since for every $k'\neq k$ we have that $d_X(A_k,A_{k'})\geq 2(k+k')$, and $k>r$, the sets $B_{r+k}(A_k)$ and $\bigcup_{m\neq k} B_{m}(A_m)$ are disjoint, which implies that 
\[
g_M(x)=g_{A_k,k}(x) \text{ and }g_M(x')=g_{A_k,k}(x').
\]
Since once again $d_X(x'x)\leq r$, we have that 
\[
|g_M(x)-g_M(x')|=|g_{A_k,k}(x)-g_{A_k,k}(x')|\leq \frac{r}{k}< \frac{r}{n}<\delta.
\]
This shows that $g_M$ is slowly oscillating.
\end{proof}
We have shown that the sequence $(g_{A_n,n})_{n\in\N}$ satisfies the hypotheses of Proposition~\ref{prop:Higson1}, and therefore for all sufficiently large $n$, one has that
\[
\norm{\mathbb E_Y(\Phi(g_{A_n,n}))-\Phi(g_{A_n,n})}<\varepsilon.
\]
This is a contradiction.
\end{proof}

\section{The main result}\label{S.mainproof}
Here we prove our main result, restated for convenience.

\begin{theorem}\label{thm:main}
Let $(X,d_X)$ and $(Y,d_Y)$ be u.l.f.\ metric spaces whose uniform Roe algebras are isomorphic. Then $X$ and $Y$ are bijectively coarsely equivalent.
\end{theorem}
The same reasoning as in the proof of Theorem~\ref{thm:main} applies to quasi-local algebras. 

\begin{theorem}\label{thm:mainQL}
Let $(X,d_X)$ and $(Y,d_Y)$ be u.l.f.\ metric spaces whose quasi-local algebras are isomorphic. Then $X$ and $Y$ are bijectively coarsely equivalent.
\end{theorem}

For the proof of Theorem~\ref{thm:mainQL}, simply replace every instance of $\cstu(X)$ by $\cstql(X)$ in every step of the proof below.

Let us fix for the rest of the section two u.l.f.\ metric spaces spaces $(X,d_X)$ and $(Y,d_Y)$, and a $^*$-isomorphism 
\[
\Phi\colon\cstu(X)\to\cstu(Y).
\]
As sketched in \S\ref{ss:strat}, the idea is to construct functions 
\[\alpha\colon X\to \Fin(Y)\text{ and }\beta\colon Y\to \Fin (X)
\] 
such that 
\begin{enumerate}[label=(Bij\arabic*)]
    \item\label{Bijc1} if $f\colon X\to Y$ and $g\colon Y\to X$ are functions such that for every $x\in X$ and $y\in Y$
    \[
    f(x)\in\alpha(x)\text{ and }g(y)\in\beta(y)
    \]
    then $f$ and $g$ are mutually inverse coarse equivalences, and
    \item\label{Bijc2} for every $A\in \Fin(X)$ and $B\in\Fin(Y)$ we have that
\[
|A|\leq \Big|\bigcup_{x\in A}\alpha(x)\Big|\text{ and } |B|\leq \Big|\bigcup_{y\in B}\beta(y)\Big|.
\]
\end{enumerate}
Note that condition \ref{Bijc1} implies that each of the sets $\alpha(x)$ and $\beta(y)$ is nonempty, and that any two functions $f,f'\colon X\to Y$ such that $f(x),f'(x)\in\alpha(x)$ for all $x\in X$ must be close to each other. Let us record Hall's marriage theorem.

\begin{theorem}[Hall, \cite{Hall}]
Let $X,Y$ be sets and let $\alpha \colon X \to \Fin(Y)$ be a function. There is an injective $f\colon X\to Y$ such that $f(x)\in \alpha(x)$ for all $x\in X$ if and only if for every $A\in\Fin(X)$ we have 
\[
|A|\le \Big|\bigcup_{x\in A}\alpha(x)\Big|.
\]
 \label{ThmHall}
\end{theorem}

\begin{lemma}\label{lemma:alphabeta}
Suppose that $\alpha\colon X\to \Fin(Y)$ and $\beta\colon Y\to\Fin(X)$ are functions satisfying conditions \ref{Bijc1} and \ref{Bijc2}. Then any $f\colon X\to Y$ such that $f(x)\in\alpha(x)$ for all $x\in X$ is close to a bijective coarse equivalence.
\end{lemma}
\begin{proof}
Since each $\alpha(x)$ is nonempty, there is a function $f\colon X\to Y$  such that $f(x)\in\alpha(x)$ for all $x\in X$. By Theorem~\ref{ThmHall} and condition~\ref{Bijc2} we can assume that $f$ is injective. Similarly, we can construct an injective $g\colon Y\to X$ such that $g(y)\in\beta(y)$ for all $y\in Y$. K\"onig's proof of the Cantor--Schroeder--Bernstein theorem gives a bijection $h\colon X\to Y$ such that, for all $x\in X$, we have that either $h(x)=f(x)$ or $x\in \mathrm{Im}(g)$ and $h(x)=g^{-1}(x)$. Since $f$ and $g$ are mutually inverse coarse equivalences, $h$ is close to $f$. This is the required bijective coarse equivalence.
\end{proof}
Our next goal is to construct candidates for the functions $\alpha$ and $\beta$.
For $x\in X$, $y\in Y$, and $\varepsilon>0$, we let
\[
Y_{x,\varepsilon}=\{z\in Y\mid \norm{\Phi(\chi_x)\chi_z}>\varepsilon\} \text{ and } X_{y,\varepsilon}=\{w\in X\mid \norm{\Phi^{-1}(\chi_y)\chi_w}>\varepsilon\}.
\]
 If $A\subseteq X$ and $B\subseteq Y$ we let
 \[
 Y_{A,\varepsilon}=\bigcup_{x\in A}Y_{x,\varepsilon} \text{ and }X_{B,\varepsilon}=\bigcup_{y\in B}X_{y,\varepsilon}.
 \]
For a proof of the following lemma, see \cite[Theorem 4.1]{SpakulaWillett2013AdvMath} or \cite[Proposition 1.10]{SquareInventiones}.

\begin{lemma}\label{lem:JanRufus}
Let $m\geq 0$ and $\varepsilon>0$. Then each $Y_{x,\varepsilon}$ (and each $X_{y,\varepsilon}$) is finite. In addition, suppose that $f\colon X\to Y$ is such that $f(x)\in B_m(Y_{x,\varepsilon})$ for all $x\in X$. Then $f$ is a coarse equivalence, and if $g\colon Y\to X$ is  such that $g(y)\in B_m(X_{y,\varepsilon})$ for all $y\in Y$, then $f$ and $g$ are mutually inverse coarse equivalences.
\end{lemma}

All strategies to obtain coarse equivalences from isomorphisms of uniform Roe algebras follow the approach of \cite{SpakulaWillett2013AdvMath}, who showed (in the property A setting) the existence of $\varepsilon>0$ such that all sets of the form $Y_{x,\varepsilon}$ are nonempty, and then applied Lemma~\ref{lem:JanRufus}. In \cite[Lemma 3.2]{SquareInventiones} the following stronger result (which does not need any geometric assumption) was proved.

\begin{lemma}\label{lem:Inv}
For every $\delta>0$ there is $\varepsilon>0$ such that for every  $x\in X$ and $y\in Y$ we have that 
\[
\max \left\{\norm{(1-\chi_{Y_{x,\varepsilon}})\Phi(\chi_x)},\norm{(1-\chi_{X_{y,\varepsilon}})\Phi^{-1}(\chi_y)}\right\}<\delta.\]
Consequently for every $F\in\Fin(X)$ and $\delta>0$ we can find $\varepsilon>0$ such that 
\begin{equation}\label{eqn1}
\norm{(1-\chi_{Y_{F,\varepsilon}})\Phi(\chi_F)}<\delta.
\end{equation}
\end{lemma}

 For a given $\delta>0$, we would like find a single $\varepsilon>0$ that makes \eqref{eqn1} true independently of the finite set $F$. This is formalised by the following condition, for $\delta\in (0,1)$:

\begin{align*}\label{thegoal}\tag{GOAL($\delta$)}
&\exists \varepsilon>0 \exists m\in\mathbb N \text{ s.t. }\forall A\in \Fin(X)\ \,\forall B\in\Fin(Y) \\
&\max \{\norm{(1-\chi_{B_m(Y_{A,\varepsilon})})\Phi(\chi_A)}, \norm{(1-\chi_{B_m(X_{B,\varepsilon})})\Phi^{-1}(\chi_B)}\}< \delta.
\end{align*}
This is the norm estimate we mentioned in \S\ref{ss:strat}.
\begin{lemma}\label{lemma:hall}
Suppose \eqref{thegoal} holds for some $\delta\in (0,1)$, as witnessed by $\varepsilon$ and $m$. Then for all $A\in\Fin(X)$ and $B\in\Fin(Y)$ we have that
\[
|A|\leq \Big|\bigcup_{x\in A}B_m(Y_{x,\varepsilon})\Big| \text{ and } |B|\leq \Big|\bigcup_{y\in B}B_m(X_{y,\varepsilon})\Big|.
\]
Consequently, if $f\colon X\to Y$ is a function such that $f(x)\in B_m(Y_{x,\varepsilon})$ for all $x\in X$ then $f$ is close to a bijective coarse equivalence.
\end{lemma}
\begin{proof}
We reason as in \cite[Lemma 6.8]{BragaFarahVignati2019} (see also \cite[Lemma 6.9]{WhiteWillett2017}). Fix $A\in \Fin(X)$, and suppose that $\norm{(1-\chi_{B_m(Y_{A,\varepsilon})})\Phi(\chi_A)}<\delta$. Note that
\[
B_m(Y_{A,\varepsilon})=\bigcup_{x\in A}B_m(Y_{x,\varepsilon}).
\]
Assume for a contradiction that $|A|>|B_m(Y_{A,\varepsilon})|$. Since $\Phi$ is rank preserving, \[
|A|=\rank(\chi_A)=\rank(\Phi(\chi_A))\text{ and }|B_m(Y_{A,\varepsilon})|=\rank(\chi_{B_m(Y_{A,\varepsilon})}).
\]
If $|A|>|B_m(Y_{A,\varepsilon})|$, then 
\[
\rank(\Phi(\chi_A))> \rank(\chi_{B_m(Y_{A,\varepsilon})})=\corank(1-\chi_{B_m(Y_{A,\varepsilon})}). 
\]
We can then find a  unit vector $\xi\in \Img(1-\chi_{B_m(Y_{A,\varepsilon})})\cap\Img(\Phi(\chi_A))$, and so 
\[
1=\norm{(1-\chi_{B_m(Y_{A,\varepsilon})})\Phi(\chi_A)\xi}\leq \norm{(1-\chi_{B_m(Y_{A,\varepsilon})})\Phi(\chi_A)}<\delta.
\]
This is a contradiction. The same exact proof gives that $|B|\leq |B_m(X_{B,\varepsilon})|$ for any given $B\in\Fin(Y)$.

The last statement follows from Lemma~\ref{lemma:alphabeta} applied to the functions $\alpha$ and $\beta$ defined by \[
\alpha(x)=B_m(Y_{x,\varepsilon}) \text{ and }\beta(y)=B_m(X_{y,\varepsilon}).\qedhere
\]
\end{proof}

The rest of this section is dedicated to prove \eqref{thegoal} for some $\delta\in(0,1)$.

\begin{lemma}\label{lem:deltahalf}
Let $\delta\in (0,1)$. If \eqref{thegoal} fails, then 
\begin{enumerate}[label=(F\arabic*)]
    \item \label{Fin1}
for every $F\in\Fin(X)$, every $\varepsilon>0$ and every $m>0$ there is $A\in\Fin(X\setminus F)$ such that 
\[
\norm{(1-\chi_{B_m(Y_{A,\varepsilon})})\Phi(\chi_{A})}>\delta/2,
\]
or
\item \label{Fin2} for every $G\in\Fin(Y)$, every $\varepsilon>0$ and every $m>0$ there is $B\in\Fin(Y\setminus G)$ such that 
\[
\norm{(1-\chi_{B_m(X_{B,\varepsilon})})\Phi^{-1}(\chi_{B})}>\delta/2.
\]
\end{enumerate}
\end{lemma}
\begin{proof}
Assume \ref{Fin1} fails and fix $F\in\Fin(X)$, $\varepsilon>0$ and $m\in\mathbb N$ such that for every $A\in\Fin(X\setminus F)$ we have
\[
\norm{(1-\chi_{B_m(Y_{A,\varepsilon})})\Phi(\chi_{A})}\leq\delta/2.
\]
By Lemma~\ref{lem:Inv}, we can find $\gamma\in (0,\varepsilon)$ such that $\norm{(1-\chi_{Y_{F',\gamma}})\Phi(\chi_{F'})}<\delta/2$ for every $F'\subseteq F$. Pick now an arbitrary $C\in\Fin(X)$. Note that if $C'\subseteq C$, $k'\leq k$ and $\eta\leq\eta'$ then $B_{k'}(Y_{C',\eta'})\subseteq B_k(Y_{C,\eta})$ and so 
\[
1-\chi_{B_k(Y_{C,\eta})}\leq 1-\chi_{B_{k'}(Y_{C',\eta'})}.
\]
Then
\begin{eqnarray*}
\norm{(1-\chi_{B_m(Y_{C,\gamma})})\Phi(\chi_C)}&\leq& \norm{(1-\chi_{B_m(Y_{C,\gamma})})\Phi(\chi_{C\cap F})}+\norm{(1-\chi_{B_m(Y_{C,\gamma})})\Phi(\chi_{C\setminus F})}\\
&\leq&\norm{(1-\chi_{B_m(Y_{C\cap F,\gamma})})\Phi(\chi_{C\cap F})}+\norm{(1-\chi_{B_m(Y_{C\setminus F,\gamma})})\Phi(\chi_{C\setminus F})}\\
&\leq&\delta/2+\norm{(1-\chi_{B_m(Y_{C\setminus F,\varepsilon})})\Phi(\chi_{C\setminus F})}\\
&<&\delta/2+\delta/2=\delta.
\end{eqnarray*}
The same exact argument gives that if \ref{Fin2} fails then we can find $\gamma>0$ and $m$ such that for every $D\in \Fin(Y)$ we have that 
\[
\norm{(1-\chi_{B_m(X_{D,\gamma})})\Phi^{-1}(\chi_D)}.
\]
We have shown that if both \ref{Fin1} and \ref{Fin2} fail, \eqref{thegoal} holds. Contrapositively, if \eqref{thegoal}, (at least) one of \ref{Fin1} and \ref{Fin2} must hold.
\end{proof}

\begin{lemma}\label{lemma:itholds}
For every $\delta\in (0,1)$, \eqref{thegoal} holds.
\end{lemma}
\begin{proof}
Let $\delta\in (0,1)$, and suppose that \eqref{thegoal} fails. Without loss of generality, we assume condition \ref{Fin1} of Lemma~\ref{lem:deltahalf} holds. (In case condition \ref{Fin2} holds, we repeat the proof using $\Phi^{-1}$ instead of $\Phi$.)

The functions $g_{A,r}$, for $A\subseteq X$ and $r\geq 0$, were defined just before Proposition~\ref{prop:Higson2}. Applying Proposition~\ref{prop:Higson2}, we can find $r\geq 0$ and  $F\in\Fin(X)$ such that for every $A\in\Fin(X\setminus F)$
\begin{equation}\label{eqn2}
\norm{\mathbb E_Y(\Phi(g_{A,r}))-\Phi(g_{A,r})}<\delta/8    
\end{equation}
Fix now $\varepsilon>0$ small enough such that $\norm{(1-\chi_{X_{y,\varepsilon}})\Phi^{-1}(\chi_y)}<\delta/4$ for all $y\in Y$. Further, let $m$ be such that if $x,x'\in X$ are such that $d_X(x,x')\leq r$, then for every  $y\in Y_{x,\varepsilon}$ and $y'\in Y_{x',\varepsilon}$ one has that $d_Y(y,y')\leq m$. Such an $\varepsilon$ exists by Lemma~\ref{lem:Inv}, and the existence of $m$ is granted by the fact that every association $x\mapsto y\in Y_{x,\varepsilon}$ is coarse (Lemma~\ref{lem:JanRufus}).

Since condition \ref{Fin1} of  Lemma~\ref{lem:deltahalf} holds, we can find $A\in\Fin(X\setminus B_r(F))$ such that 
\[
\norm{(1-\chi_{B_m(Y_{A,\varepsilon})})\Phi(\chi_{A})}>\delta/2.
\]
Since $g_{A,r}\chi_{A}=\chi_{A}$, then
\[
\delta/2<\norm{(1-\chi_{B_m(Y_{A,\varepsilon})})\Phi(g_{A,r})\Phi(\chi_{A})}\leq \norm{(1-\chi_{B_m(Y_{A,\varepsilon})})\Phi(g_{A,r})}.
\]
By Equation~\eqref{eqn2} we have that
\[
\norm{(1-\chi_{B_m(Y_{A,\varepsilon})})\mathbb E_Y(\Phi(g_{A,r}))}>\frac{3\delta}{8}.
\]
As both $(1-\chi_{B_m(Y_{A,\varepsilon})})$ and $\mathbb E_Y(\Phi(g_{A,r}))$ belong to $\ell_\infty(Y)$, where the norm is given by the sup norm, we can find $y\in Y\setminus B_m(Y_{A,\varepsilon})$ such that $\norm{\mathbb E_Y(\Phi(g_{A,r}))\chi_y}>3\delta/8$. Again applying Equation~\eqref{eqn2} we obtain that $\norm{\Phi(g_{A,r})\chi_y}>\delta/4$. Since $\Phi^{-1}$ is an isometry,
\[
\norm{g_{A,r}\Phi^{-1}(\chi_y)}=\norm{\Phi(g_{A,r})\chi_y}>\delta/4.
\]
Since $g_{A,r}$ is supported on $B_r(A)$, we have that $g_{A,r}\chi_{B_r(A)}=g_{A,r}$ and consequently
\[
\norm{\chi_{B_r(A)}\Phi^{-1}(\chi_y)}>\delta/4.
\]
Since our choice of $\varepsilon$ gives that $\norm{(1-\chi_{X_{y,\varepsilon}})\Phi^{-1}(\chi_y)}<\delta/4$, $B_r(A)$ must intersect $X_{y,\varepsilon}$. Pick 
\[
x\in B_r(A)\cap X_{y,\varepsilon},
\]
and note that, by symmetry, $y\in Y_{x,\varepsilon}$. Let $x'\in A$ with $d_X(x,x')\leq r$, and let $y'\in Y_{x',\varepsilon}\subseteq Y_{A,\varepsilon}$. By our choice of $m$, then $d_Y(y,y')\leq m$, and consequently $y\in B_m(Y_{A,\varepsilon})$. This is a contradiction.
\end{proof}
\begin{proof}[Proof of Theorems~\ref{thm:main} and~\ref{thm:mainQL}]
By Lemma~\ref{lemma:hall}, it is enough to show that \eqref{thegoal} holds for some $\delta\in (0,1)$. This is the thesis of Lemma~\ref{lemma:itholds}
\end{proof}

\section{Corollaries and concluding remarks}\label{S.CorConcl}

Here we prove Theorems~\ref{thmi:main2} and \ref{thmi:main3}, two corollaries of the proof of Theorem~\ref{thm:main} on the uniqueness of certain Cartan subalgebras of uniform Roe algebras and automorphisms thereof.

\subsection{Roe Cartan subalgebras}
As mentioned uniqueness of Cartan masas does not hold in uniform Roe algebras, as there might exist \emph{exotic} Cartan masas which are not isomorphic to $\ell_\infty(\N)$ (see \cite[\S3]{WhiteWillett2017}). Therefore, to obtain uniqueness result, we need to focus on specific Cartan masas.

For an inclusion of \cstar-algebras $A\subseteq B$ we say that $A$ is co-separable in $B$ if there is a countable $S\subseteq B$ such that $B=\mathrm{C}^*(A,S)$, meaning that $A$ and $S$ generated $B$ as a \cstar-algebra. 
\begin{definition}\label{def:RoeCartan}
Let $X$ be a u.l.f.\ metric space. A \cstar-subalgebra $A\subseteq\cstu(X)$ is said to be a Roe Cartan masa if $A$ is a co-separable Cartan masa in $\cstu(X)$ which is isomorphic to $\ell_\infty(\N)$.
\end{definition}

The following `uniqueness of Roe Cartan subalgebras' result was proved as Theorem E in \cite{WhiteWillett2017}:
\begin{theorem}\label{thm:WWCartan}
Let $X$ be a u.l.f.\ property A metric space, and let $A\subseteq\cstu(X)$ be a Roe Cartan masa. Then there is a unitary $u\in\cstu(X)$ such that $A=u\ell_\infty(X)u^*$.
\end{theorem}

We extend Theorem~\ref{thm:WWCartan} outside the property A scope, by essentially re-running the proof of \cite{WhiteWillett2017}. There, White and Willett use bijective coarse rigidity to construct the desired unitary $u
$ and prove that  $u$ is a quasi-local operators. Since in the property A setting $\cstql(X)$ and $\cstu(X)$ coincide, then $u\in\cstu(X)$. Using Theorem~\ref{thm:main}, one can still construct the unitary $u$ and prove that it is quasi-local, exactly as in \cite{WhiteWillett2017}. Yet, since outside of the property A setting the inclusion $\cstu(X)\subseteq\cstql(X)$ might be proper (this is the main result of \cite{Ozawa2023uRaSmallerQL}), to conclude that the desired unitary belongs to $\cstu(X)$ an additional argument is needed. For this, we slightly generalise a result of Mart\'inez and Vigolo from \cite{MartinezVigoloLong}.

Let $(X,d)$ be a u.l.f.\ metric space. If $H$ is a separable infinite-dimensional Hilbert space, operators in $\mathcal B(\ell_2(X,H))$ can be viewed as $X$-by-$X$ matrices valued in $\mathcal B(H)$. The propagation of an operator $a=[a_{x,x'}]_{x,x'\in X}$ is again the quantity
\[
\propg(a)=\sup \{d(x,x')\mid a_{x,x'}\neq 0\}.
\]
The \emph{algebra of banded operators}, denoted $\mathrm{BD}(X)$, is the closure of the algebra of finite propagation operators. The \emph{Roe algebra} of $X$, denoted $\mathrm{C}^*(X)$, is the subalgebra of $\mathrm{BD}(X)$ given by locally compact operators, meaning that $a=[a_{x,x'}]$ belongs to $\mathrm{C}^*(X)$ if and only if $a\in \mathrm{BD}(X)$ and each $a_{x,x'}$ belongs to $\mathcal K(H)$. As noticed in \cite[Theorem 4.1]{BragaVignatiGelfand}, $\mathrm{BD}(X)$ is the multiplier algebra of $\mathrm{C}^*(X)$.

For $A\subseteq X$, we denote again by $\chi_A$ the (infinite-dimensional) projection onto $\ell_2(A,H)$. When viewed as an $X$-by-$X$ matrix, 
\[
(\chi_A)_{x,x'}=\begin{cases}
    1_{\mathcal B(H)} &\text{ if } x=x'\in A\\
    0 & \text{ else}.
\end{cases}
\]
We say that an operator $a\in\mathcal B(\ell_2(X,H))$ is quasi-local if for every $\varepsilon>0$ there is $r>0$ such that $\norm{\chi_Aa\chi_B}<\varepsilon$ for every $A,B\subseteq X$ with $d(A,B)>r$. The algebra of quasi-local operators on $\ell_2(X,H)$ is denoted by $\mathrm{C}^*_{qlBD}(X)$.

In \cite[Corollary 11.4.10]{MartinezVigoloLong}, Mart\'inez and Vigolo showed the following:
\begin{theorem}\label{thm:MVQL}
Let $X$ be a u.l.f.\ metric space, and suppose that $u\in \mathrm{C}^*_{qlBD}(X)\setminus\mathrm{BD}(X)$ is a unitary. Then the automorphism of $\mathrm{C}_{qlBD}^*(X)$ given by $a\mapsto uau^*$ does not send $\mathrm{BD}(X)$ to itself. 
\end{theorem}

We extend this result to the uniform setting. The following was already noticed in \cite[Remark 3.1]{BragaSpakulaVignati2025}.

\begin{prop}\label{prop:inner}
Let $X$ be a u.l.f.\ metric space, and suppose that $u$ is a unitary in $\cstql(X)\setminus \cstu(X)$. Then the automorphism of $\mathrm{C}_{ql}^*(X)$ given by $a\mapsto uau^*$ does not send $\cstu(X)$ to itself.
\end{prop}
\begin{proof}
Fix a unitary $u\in\cstql(X)\setminus \cstu(X)$, and suppose that $Ad(u)[\cstu(X)]=\cstu(X)$. Let $v=u\otimes 1_{\mathcal B(H)}$, so that $v$ defines an automorphism of the Roe algebra $\mathrm{C}^*(X)$, meaning that $Ad(v)[\mathrm{C}^*(X)]=\mathrm{C}^*(X)$. Since $\mathrm{BD}(X)$ is the multiplier algebra of $\mathrm{C}^*(X)$ (\cite[Theorem 4.1]{BragaVignatiGelfand}), $Ad(v)$ maps $\mathrm{BD}(X)$ to $\mathrm{BD}(X)$.
\begin{claim}\label{claim:expanding}
    $v\in\mathrm{C}^*_{qlBD}(X)\setminus \mathrm{BD}(X)$.
\end{claim}
\begin{proof}
Let $\psi$ be a normal state on $\mathcal B(H)$, and consider the slice map \[
L_\psi\colon \mathcal B(\ell_2(X,H))=\mathcal B(\ell_2(X))\bar\otimes\mathcal B(H)\to\mathcal B(\ell_2(X)).
\]
$L_\psi$ is a conditional expectation (see \cite[III.2.2.6]{Blackadar.OA}) onto $\mathcal B(\ell_2(X))$. As for an operator in $\mathcal B(\ell_2(X,H))=\mathcal B(\ell_2(X))\bar\otimes\mathcal B(H)$ the property of `having finite propagation' depends only on $\mathcal B(\ell_2(X))$, $L_\psi$ maps finite propagation operators to finite propagation operators. Suppose now that $v\in\mathrm{BD}(X)$, fix $\varepsilon>0$, and let $v'\in\mathcal B(\ell_2(X,H))$ be a finite propagation operator with $\norm{v'-v}<\varepsilon$. Then $\norm{L_\psi(v')-L_\psi(v)}<\varepsilon$. As $L_{\psi}(v')\in\mathcal B(\ell_2(X))$ has finite propagation and $L_\psi(v)=u$, we have shown that $u$ can be $\varepsilon$ approximated by a finite propagation operator in $\mathcal B(\ell_2(X))$. As $\varepsilon$ is arbitrarily small, this shows that $u\in\cstu(X)$. This is a contradiction.
\end{proof}
By Theorem~\ref{thm:MVQL} and Claim~\ref{claim:expanding}, $v\mathrm{BD}(X)v^*$ is not contained in $\mathrm{BD}(X)$. This is a contradiction.
\end{proof}

We are ready to extend Theorem~\ref{thm:WWCartan}.

\begin{theorem}\label{thm:Cartan}
Let $X$ be a u.l.f.\ metric space and suppose that $A\subseteq \cstu(X)$ is a Roe Cartan subalgebra. Then there is a unitary $v\in\cstu(X)$ such that $vAv^*=\ell_\infty(X)$.
\end{theorem}

\begin{proof}
We reason as in \S6 of \cite{WhiteWillett2017}. Let $A\subseteq \cstu(X)$ be a Roe Cartan subalgebra. By Theorem B of \cite{WhiteWillett2017}, there is a u.l.f.\ metric space $Y$ and a unitary $u\colon \ell_2(Y)\to\ell_2(X)$ such that 
\[
u\ell_\infty(Y)u^*=A \text{ and }u\cstu(Y)u^*=\cstu(X).
\]
Let $f\colon X\to Y$ be a bijective coarse equivalence constructed as in the proof of Theorem~\ref{thm:main}. We let $w\colon\ell_2(X)\to \ell_2(Y)$ be the unitary defined by $w\delta_x=\delta_{f(x)}$. Note that since $f$ is a bijective coarse equivalence $w\cstu(X)w^*=\cstu(X)$.
Set
\[
v=w^*u^*.
\]
We claim that $v$ is as required. Note that 
\[
vAv^*=w^*u^*Auw=w^*\ell_\infty(Y)w=\ell_\infty(X).
\]
Following exactly the same proof as in Theorem E in \cite{WhiteWillett2017}, we have that $v$ belongs to $\cstql(X)$. Since
\[
v\cstu(X)v^*=w^*u^*\cstu(X)uw=w^*\cstu(Y)w=\cstu(X),
\]
$\mathrm{Ad}(v)$ induces an automorphism of $\cstu(X)$. The contrapositive of Proposition~\ref{prop:inner} shows that $v\in\cstu(X)$, and this concludes the proof.
\end{proof}

We do not know whether the requirement of co-separability in the definition of Roe Cartan pair is necessary, or whether the latter is already automatic. The strongest result in this direction is \cite[Theorem 1.12]{SquareInventiones} asserting that if $\cstu(X)$ and $\cstu(Y)$ are isomorphic, both spaces are uniformly locally finite, and $X$ is metrizable, then $Y$ contains a coarse copy of $X$. For more on rigidity of uniform Roe algebras associated to general coarse structures, see \S\ref{ss.remarks}.

\subsection{Automorphisms of uniform Roe algebras}

A consequence of Theorem~\ref{thm:WWCartan} is a Gelfand duality type of result for uniform Roe algebras, which we now describe. 

For a u.l.f.\ metric space $X$, we denote by $\mathrm{BijCoa}(X)$ the group of bijective coarse equivalences of $X$ modulo closeness. If $f\colon X\to X$ is a bijective coarse equivalence, the unitary $u_f\colon \ell_2(X)\to \ell_2(X)$ defined by $u_f\delta_x=\delta_{f(x)}$ gives an automorphism $Ad(u_f)\colon \cstu(X)\to\cstu(X)$. If $f$ and $f'$ are two close bijective coarse equivalences, then the corresponding automorphisms are conjugated by the inner automorphism of $\cstu(X)$ given by the unitary which sends $\delta_{f(x)}$ to $\delta_{f'(x)}$. It is easy to verify that $u_f\in\cstu(X)$ if and only if $f$ is close to the identity. This association gives a canonical injective homomorphism 
\[
\mathrm{BijCoa}(X)\to\mathrm{Out}(\cstu(X)),
\]
where $\mathrm{Out}(\cstu(X))$ is the group of automorphisms of $\cstu(X)$ modulo inner ones. Theorem A in \cite{BragaVignatiGelfand} asserts that in case of property A such canonical map is an isomorphism. We extend this result.

\begin{theorem}\label{thm:automs}
Let $X$ be a u.l.f.\ metric space. Then the canonical homomorphism
\[
\mathrm{BijCoa}(X)\to\mathrm{Out}(\cstu(X))
\]
is an isomorphism.
\end{theorem}
\begin{proof}
We follow the proof of Theorem A in \cite{BragaVignatiGelfand}. We need to prove that the canonical injective homomorphism is surjective. Let $\Phi \in \Aut(\cstu(X))$, and let $A=\Phi[\ell_\infty(X)]$, so that $A\subseteq\cstu(X)$ is a Roe Cartan subalgebra. By Theorem~\ref{thm:Cartan} we can find a unitary $u \in \cstu(X)$ so that $\Psi=\Ad(u)\circ \Phi$ is an automorphism of $\cstu(X)$ which takes $\ell_\infty(X)$ to itself. As every automorphism of $\cstu(X)$ is spatial (see \cite[Lemma 3.1]{SpakulaWillett2013AdvMath}), we can find a unitary $v\in\mathcal B(\ell_2(X))$ such that $\Psi=\Ad(v)$. As $v\ell_\infty(X)v^*=\ell_\infty(X)$, there is a bijective coarse equivalence $f\colon X\to X$ and a family $(\lambda_x)_{x\in X}$ in the unit circle of $\mathbb C$ so that $v\delta_x=\lambda_x\delta_{f(x)}$ for all $x\in X$ (see Lemma 8.10 and the proof of Theorem 8.1 in \cite{BragaFarah2018}). Hence, $\mathrm{Ad}(v_f)$ equals $\Psi$ modulo $\mathrm{Inn}(\cstu(X))$, which in turn, as $u\in \cstu(X)$, equals $\Phi$ modulo $\mathrm{Inn}(\cstu(X))$.
\end{proof}

\subsection{Concluding remarks}\label{ss.remarks}

We collect a list of remarks and open questions.

\subsubsection{A self-contained proof of bijective rigidity}

Even though our argument relies heavily on the already paved road to rigidity (using for example Lemma 3.2 of \cite{SquareInventiones}), it would be possible to give a completely self-contained proof of Theorem~\ref{thmi:main}, still following the strategy of Willett and \v{S}pakula (\cite{SpakulaWillett2013AdvMath}) and White and Willett (\cite{WhiteWillett2017}). Such an argument would pass by a modification of the sets $X_{x,\varepsilon}$. Fix $X$, $Y$, and $\Phi$ as in \S\ref{S.mainproof}, and let $r\geq 0$. Let $g_{x,r}$ be the `flattened' indicator function at $x$ as introduced before Proposition~\ref{prop:Higson2}. Defining 
\[
\alpha_r(x)=\{y\in Y\mid \norm{\Phi(g_{x,r})\chi_y}>1/2\},
\]
it is possible, by a technical diagonalisation argument relying on Proposition~\ref{prop:Higson2}, to find $r\geq 0$ such that 
\[
|A|\leq \Big|\bigcup_{x\in A}\alpha_r(x)\Big|
\]
for all $A\in \Fin(X)$. Symmetrically, one constructs functions $\beta_r\colon Y\to\Fin(X)$ satisfying the analogue of the above inequality, and then continues as in Lemma~\ref{lemma:alphabeta}. The arguments required to prove that one can find a large enough $r$ such that $\alpha_r$ and $\beta_r$ satisfy the hypotheses of Lemma~\ref{lemma:alphabeta} turn out to be quite technical, slightly unpleasant, and not particularly enlightening. For this reason, we decide not to include them, especially considering the neatness of the proof of Lemma~\ref{lem:Inv} contained in  \cite{MartinezVigolo}.

\subsubsection{General coarse structures}
Coarse spaces are generalisations to the uncountable of the coarse approach to metric spaces. To be precise, if $X$ is a set, some $\mathcal E\subseteq\mathcal P(X\times X)$ is a \emph{coarse structure} on $X$ if
\begin{itemize}
\item the diagonal $\Delta_X=\{(x,x)\in X\times X\mid x\in X\}\in \cE$,
\item if $E\in \cE$, then $E^{-1}=\{(x,y)\in X\times X\mid (y,x)\in E\}\in \cE$,
\item if $E\in \cE$ and $F\subseteq E$, then $F\in \cE$, 
\item if $E,F\in \cE$, then $E\cup F\in \cE$, and
\item if $E,F\in \cE$, then $E\circ F\in\mathcal E$, where
\[
E\circ F=\{(x,z)\in X\times X\mid\exists y\in X,\ (x,y)\in E\wedge (y,z)\in F \}.
\] 
\end{itemize}
Elements of $\mathcal E$ are called \emph{entourages}, and the pair $(X,\mathcal E)$ is called a \emph{coarse space}. A coarse space $(X,\cE)$ is 
\begin{itemize}
    \item connected if it contains all finite subsets of $X\times X$,
    \item uniformly locally finite (u.l.f.\ ) if for all $E\in\mathcal E$ we have that 
\[
\sup_{x\in X}\{y\in X\mid (x,y)\in E\}<\infty, \text{ and }
\]
\item countably generated if there is a countable $S\subseteq\mathcal E$ such that $\mathcal E$ is the smallest coarse structure containing $S$.
\end{itemize}

The definition of coarse functions and (bijective) coarse equivalence have obvious generalisations to the setting of coarse structures.

Typical examples of coarse spaces are induced by metrics. If $d$ is a metric on a set $X$, one considers the coarse structure $\mathcal E_d$ given by $d$-bounded sets, meaning that $E\in\mathcal E_d$ if and only if 
\[
\sup \{d(x,x')\mid (x,x')\in E\}<\infty.
\]
We call a coarse structure $(X,\mathcal E)$ metrizable if $\mathcal E=\mathcal E_d$ for some metric $d$ on $X$. Metrizable coarse structures are the \emph{small} objects in the coarse category. In fact, as shown in \cite[Theorem 2.55]{RoeBook}, for connected coarse structures, metrizability is equivalent to being countably generated.

If $(X,\mathcal E)$ is a u.l.f.\ coarse space, we can construct its uniform Roe algebra $\cstu(X,\mathcal E)$ by closing in norm the set of bounded linear operators on $\ell_2(X)$ which are supported on an entourage. The associated rigidity problem is still open.
\begin{problem}\label{prob:general}
    Prove that any two u.l.f.\ coarse structures $(X,\mathcal E)$ and $(Y,\mathcal F)$ with isomorphic uniform Roe algebras must be bijective coarsely equivalent.
\end{problem}
The only known (partial) solutions to Problem~\ref{prob:general} are in the property A setting, where `weak rigidity' holds (meaning that one can prove coarse equivalence from isomorphism of uniform Roe algebras). We refer to \cite{BragaFarahVignati2020AnnInstFour} and \cite[\S4]{BragaFarah2018} for more details and precise statements.

The difficulty of generalising to the general setting results valid in the metrizable setting is that, often, many of technical arguments are based on diagonalisation techniques, and we do not know whether these are replicable in the nonmetrizable setting. In fact, Higson coronas in the nonmetrizable setting have been only vaguely studied. For example, we do not know whether the Higson corona is isomorphic to the center of the uniform Roe corona for a nonmetrizable coarse space, a fact which is behind our reasoning towards bijective rigidity. We suspect that some of the arguments of \cite[\S4]{BragaFarah2018}, combined with our new techniques, could be helpful in handling rigidity problems in \emph{sufficiently small} coarse structures.

\subsubsection{Embeddings}

Gromov introduced in \cite{Gromov93} the notion of \emph{coarse embeddings} between metric spaces, to serve as the appropriate notion of injection in the coarse setting. A coarse function $f\colon X\to Y$ is a coarse embedding if $f$ is in addition \emph{expanding}, meaning that far points are sent to far points, or, more precisely, that if $g\colon\mathrm{Im}(f)\to X$ is an inverse of $f$, then $g$ is coarse. (A version of this definition suitable for general coarse structures asks for the inverse image of an entourage to be an entourage.) The task of modelling algebraically coarse embeddings between u.l.f.\ metric spaces in terms of certain embeddings between their associated uniform Roe algebras was pursued in \cite{BragaFarahVignati2019}, and most of the known results so far rely on metrizability or regularity assumptions (see e.g. Theorem 1.12 in \cite{SquareInventiones} and Theorems 1.2 and 1.4(ii) of \cite{BragaFarahVignati2019}). It is a concrete plan to develop our new techniques to obtain algebraic conditions equivalent to the existence of injective coarse embeddings between u.l.f.\ metric spaces and solve the corresponding rigidity problems for embeddings.

\bibliographystyle{amsplain}
\bibliography{bibliography}

\end{document}